\documentclass[a4paper,reqno,oneside,11pt]{amsart}
\usepackage{fullpage}

\usepackage{amsmath}
\usepackage{amsfonts}
\usepackage{amssymb}
\usepackage{verbatim}
\usepackage[colorlinks]{hyperref}
\usepackage{caption}

\usepackage{amsthm}

\newcommand{\doi}[1]{\href{http://dx.doi.org/#1}{\texttt{doi:#1}}}
\newcommand{\arxiv}[1]{\href{http://arxiv.org/abs/#1}{\texttt{arXiv:#1}}}

\newtheorem{theorem}{Theorem}[section]
\newtheorem{lemma}[theorem]{Lemma}

\newtheorem{definition}[theorem]{Definition}

\newtheorem{remark}[theorem]{Remark}
\newtheorem{problem}[theorem]{Problem}

\newcommand{\Hnp}{\mathcal{H}^{(r)}(n,p)}

\newcommand{\cH}{\mathcal{H}}
\newcommand{\cG}{\mathcal{G}}
\newcommand{\cF}{\mathcal{F}}
\newcommand{\cE}{\mathcal{E}}

\newcommand{\EE}{\mathbb{E}}

\newcommand{\eps}{\varepsilon}

\newcommand{\FnD}{\mathcal{F}^{(r)}(n,\Delta)}
\newcommand{\FnDg}{\cF^{(2)}(n,\Delta)}

\newcommand{\cK}{\mathcal{K}}

\newcommand{\dcup}{\dot{\cup}}

\newcommand{\Erm}{\cE^{(r)}(m)}
\newcommand{\tH}{\tilde{H}}

\title{On almost optimal universal hypergraphs}
\author{Samuel Hetterich, Olaf Parczyk and Yury Person}
\address{Goethe-Universit\"at, Institut f\"ur Mathematik,
  Robert-Mayer-Str. 10, 60325 Frankfurt am Main, Germany}
\email{hetterich | parczyk | person@math.uni-frankfurt.de}
\date{\today}

 \thanks{
    The research of Olaf Parczyk and Yury Person leading to these results was supported by DFG grant PE 2299/1-1. The research of Samuel Hetterich leading to these results has received funding from the European Research Council under the European Union's Seventh Framework Programme (FP/2007-2013) / ERC Grant Agreement n. 278857-PTCC}

\begin{document}

\begin{abstract}

A hypergraph $H$ is called universal for a family $\cF$ of hypergraphs, if it contains every hypergraph $F \in \cF$ as a copy.
For the family of $r$-uniform hypergraphs with maximum vertex degree bounded by $\Delta$ and at most $n$ vertices any universal hypergraph has to contain $\Omega(n^{r-r/\Delta})$ many edges.
We exploit constructions of Alon and Capalbo
 to obtain  universal $r$-uniform hypergraphs with the optimal number of edges $O(n^{r-r/\Delta})$ when $r$ is even, $r \mid \Delta$ or $\Delta=2$. Further we generalize the result of Alon and Asodi
 about optimal universal graphs for the family of graphs with at most $m$ edges and no isolated vertices to hypergraphs.
\end{abstract}

\maketitle

\section{Introduction}

Let $\FnD$ be the family of $r$-uniform hypergraphs on at most $n$ vertices and with maximum vertex degree bounded by $\Delta$.
An $r$-uniform hypergraph is called $\FnD$-universal (or universal for $\FnD$) if it contains every $F \in \FnD$ as a copy. The purpose of this paper is to show that the existence and almost optimal explicit constructions of many universal hypergraphs follow from the corresponding results about universal graphs.

 The problem of finding various universal graphs has a long history, see an excellent survey of Alon~\cite{Alon10} and the references therein. Alon, Capalbo, Kohayakawa, R\"odl,  Ruci\'nski and Szemer\'edi studied in~\cite{ACKRRS00,ACKRRS01} explicit constructions of $\FnDg$-universal graphs and the universality of the random graph $G(n,p)$ as well. Thus, in~\cite{ACKRRS01} they constructed first nearly optimal $\FnDg$-universal graphs ($\Delta\ge 3$) with $O(n)$ vertices and $O(n^{2-2/\Delta}\ln^{1+8/\Delta}n)$ edges, while it was noted by the same authors that any such universal graph has to contain $\Omega(n^{2-2/\Delta})$ edges. Notice further, that  in the case $\Delta=2$ the square of a Hamilton cycle is $\cF^{(2)}(n,\Delta)$-universal~\cite{AC07} (and thus $2n$ edges are enough in this case). In two subsequent papers, Alon and Capalbo~\cite{AC07,AC08} improved the result of~\cite{ACKRRS01} and obtained $\FnDg$-universal graphs with the optimal number $\Theta(n^{2-2/\Delta})$ of edges and only $O(n)$ vertices and also provided $\FnDg$-universal graphs on $n$ vertices with almost optimal number of edges.

\begin{theorem}[Alon and Capalbo~\cite{AC07,AC08}]
\label{theom:AC}
For any $\Delta \ge 2$ there exist explicitly constructible $\cF^{(2)}(n,\Delta)$-universal graphs on $O(n)$ vertices with $O(n^{2-2/\Delta})$ edges and on $n$ vertices with  $O(n^{2-2/\Delta} \ln^{4/\Delta} n)$ edges.
\end{theorem}

 
 Universality of random graphs has been also a subject of intensive study by various researchers.  Alon, Capalbo, Kohayakawa, R\"odl,  Ruci\'nski and Szemer\'edi proved in~\cite{ACKRRS00} that the random graph $G((1+\eps)n,p)$ is $\FnDg$-universal for $p\ge C_{\eps,\Delta} (\ln n/n)^{1/\Delta}$ a.a.s., where $C_{\eps,\Delta}$ is a constant that depends only on $\Delta$ and $\eps$.  Since then several improvements of this result have been given. So, for example, in the spanning case Dellamonica, Kohayakawa, R{\"o}dl and Ruci{\'n}ski~\cite{DKRR15} showed that $G(n,p)$ is $\FnDg$-universal  for $p\ge C_\Delta (\ln n/n)^{1/\Delta}$ (for $\Delta \ge 3$) a.a.s.,  while the  case $\Delta=2$ was covered by Kim and Lee~\cite{kim2014universality}. In the almost spanning case, Conlon, Ferber, Nenadov and \v{S}kori\'{c}~\cite{conlon2015almost} recently showed that for every $\eps>0$ and $\Delta \ge 3$   the random graph $G((1+\eps)n,p)$ is $\FnDg$-universal for $p = \omega(n^{-1/(\Delta-1)} \ln^5 n)$ a.a.s.

The study of universal graphs has been extended recently in~\cite{PP_SparseUniversal} to universal hypergraphs by the second and third author, who showed that the random $r$-uniform hypergraph $\Hnp$ is $\FnD$-universal a.a.s.\ for $p\ge C (\ln n/n)^{1/\Delta}$, where $C$ is a constant depending on $r$ and $\Delta$ only. 
On the other hand, it follows from the asymptotic number of $\Delta$-regular $r$-uniform hypergraphs on $n$ vertices, see e.g.\  Dudek, Frieze, Ruci\'{n}ski and \v{S}ileikis~\cite{dudek2013approximate}, that  \emph{any} $\FnD$-universal hypergraph must possess  $\Omega (n^{r-r/\Delta})$ edges~\cite{PP_SparseUniversal}. 
Moreover, in~\cite{PP_SparseUniversal} explicit constructions of $\FnD$-universal hypergraphs on $O(n)$ vertices with  $O(n^{r-2/\Delta})$ edges were derived from Theorem~\ref{theom:AC}, and the existence of even sparser universal hypergraphs was obtained from the results on universality of random graphs~\cite{conlon2015almost,DKRR15}. 
For example, it was shown that there exist $\FnD$-universal hypergraphs with $n$ vertices and $\Theta\left(n^{r-\frac{r}{2\Delta}}(\ln n)^{\frac{r}{2\Delta}}\right)$ edges, which shows that the best known lower and upper bounds 
are at most the multiplicative factor $n^{\frac{r}{2\Delta}}\cdot \text{polylog}(n)$ apart. 
See the summary of these results in the table below. Here and in the following the constants in the $O$-terms depend on $r$ and $\Delta$.

\begin{table}[h]
\caption{Known universal hypergraph results for $r\ge3$~\cite{PP_SparseUniversal}.}
\begin{center}
\begin{tabular}{| l | r|}
\hline
\multicolumn{2}{| l |}{Explicit constructions of $\FnD$-universal hypergraphs}\\
\hline
$O(n)$ vertices & $O(n^{r-2/\Delta})$ edges\\
\hline
$n$ vertices & $O(n^{r-2/\Delta} \ln^{4/\Delta}n)$ edges\\
\hline
\multicolumn{2}{|l | }{Existence results of $\FnD$-universal hypergraphs}\\
\hline
$n$ vertices & $\Theta\left(n^{r-\frac{r}{2\Delta}}(\ln n)^{\frac{r}{2\Delta}}\right)$ edges\\
\hline
$(1+\eps)n$ vertices & $\omega\left(n^{r-\frac{\binom{r}{2}}{(r-1)\Delta-1}}(\ln n)^{5\binom{r}{2}}\right)$ edges \\
\hline
\end{tabular}
\end{center}
\end{table}

Another family of graphs that received attention is  the family $\Erm$ of $r$-uniform hypergraphs with  at most $m$ edges and without isolated vertices. 
 Babai, Chung, Erd\H{o}s, Graham and Spencer \cite{BCEGS82} proved that any $\cE^{(2)}(m)$-universal graph must contain $\Omega(m^2/\ln^2m)$ many edges and there exists one on $O(m^2 \ln \ln m/\ln m)$ edges.
 Alon and Asodi \cite{AA02} closed this gap by proving the existence of an $\cE^{(2)}(m)$-universal graph on $O(m^2/\ln^2m)$ edges.

\subsection{New results}
We will prove the following theorem that allows to construct $r$-uniform universal hypergraphs from universal hypergraphs of smaller uniformity and we use universal graphs from~\cite{AC07,AC08} with carefully chosen parameters to provide pretty good and in many cases almost optimal $\FnD$-universal hypergraphs.

\begin{theorem}\label{theom:main} Let $r\ge 3$ and $\Delta\ge 2$ be integers. Then the following hold.
\begin{enumerate}
\item Let $r'$ be an integer. If $r'\mid r$ and there exists an $\mathcal{F}^{(r')}(n,\Delta)$-universal hypergraph with $O(n)$ vertices and $O(n^{r'-r'/\Delta})$ edges, then there exists an $\FnD$-universal hypergraph with $O(n)$ vertices and $O(n^{r-r/\Delta})$ edges. \label{case:prime_case}
\item For even $r$ there exist explicitly constructible $\FnD$-universal hypergraphs on $O(n)$ vertices with $O (n^{r-r/\Delta})$ edges and on $n$ vertices with $O (n^{r-r/\Delta} \ln^{2r/\Delta}(n) )$ edges.\label{case:even_case}
\item For odd $r$ there exist explicitly constructible $\FnD$-universal hypergraphs on $O(n)$ vertices with $O (n^{r-(r+1)/ \Delta^\prime})$ edges and on $n$ vertices with $O (n^{r-(r+1)/ \Delta^\prime} \ln^{2(r+1)/\Delta^\prime}(n) )$ edges, where $\Delta^\prime = \lceil (r+1)\Delta/r \rceil$. In particular, if $r\mid \Delta$ this leads to the almost optimal $O(n^{r-r/\Delta}\mathrm{polylog}(n))$ edges.\label{case:odd_case}
\end{enumerate}
\end{theorem}

We see that in any case the lower and upper bounds on the edge densities of optimal universal hypergraphs differ by at most a factor of  $n^{O(r^2/\Delta^2)}$.

By applying a graph decomposition result of Alon and Capalbo from~\cite{AC07} we obtain yet another case when constructed universal hypergraphs match the lower bound.
\begin{theorem}
\label{theom:sec}
Let $r$ be an integer. 
Then there exists an explicitly constructible $\cF^{(r)}(n,2)$-universal hypergraph on $O(n)$ vertices and $O(n^{r/2})$ edges.
\end{theorem}

Finally we briefly study $\Erm$-universal hypergraphs.  It can be shown  for fixed $r\ge 3$  that  any $\Erm$-universal hypergraph must contain at least $\Omega(m^r/\ln^rm)$ many edges. This can be seen by a simple counting argument as in~\cite{BCEGS82} or by counting $(r\ln m)$-regular $r$-uniform hypergraphs on $m/\ln m$ vertices as was done in the graph case  in~\cite{AA02}. We prove that the optimal existence result of Alon and Asodi gives rise to optimal $\Erm$-universal hypergraphs.

\begin{theorem}\label{thm:Erm}
There exist $\Erm$-universal hypergraphs with $O(m^r/\ln^rm)$ edges.
\end{theorem} 

\subsection{Organization of the paper}
In the next section we introduce a very useful concept of hitting graphs, which we use in Section~\ref{sec:proof1} to prove 
Theorem~\ref{theom:main} and in Section~\ref{sec:proof2} along with a graph decomposition result from \cite{AC07} to prove Theorem~\ref{theom:sec}. In the last section we will discuss $\Erm$-universal hypergraphs and prove Theorem~\ref{thm:Erm}. We make no effort in optimizing the constants depending on $r$ and $\Delta$ hidden in the $O$-notation.

\section{Hitting graphs}
Here we define a concept of hitting graphs first introduced in~\cite{PP_SparseUniversal}. This will allow us later to obtain $r$-uniform universal hypergraphs out of universal hypergraphs of smaller uniformity. 

Let $r \ge 3$ and $2 \le s <r$ be integers.
Given two $s$-uniform hypergraphs $G$ and $F$ and an $r$-uniform hypergraph $H$, we say that $G$ \emph{hits $H$ on $F$} if for all edges $f \in E(H)$ there is a copy of $F$ in $G$ induced on $f$, i.e.\ in $G[f]$. 
A family of $s$-uniform hypergraphs $\cG$ \emph{hits} a family of $r$-uniform hypegraphs $\cF$ \emph{on} $F$ if for every $H \in \cF$ there is an $G \in \cG$ such that $G$ hits $H$ on $F$. 

This concept allows us to reduce the uniformity from $r$ to $s$ keeping at the same time much of the information about $H$. 
This motivates a definition that allows us to recover all the edges of the hypergraph $H$ which is being hit by $G$ on $F$. For given $s$-uniform hypergraphs  $G$ and $F$ let $\cH_{(F,r)}(G)$ be the $r$-uniform hypergraph  on the vertex set $V(G)$ whose edges $f \in \binom{V(G)}{r}$ are such that a copy of $F$ is contained in $G[f]$. 

The following lemma establishes the connection between hitting hypergraphs and $\cH_{(F,r)}(G)$. It is an extension of Lemma~5.2 from~\cite{PP_SparseUniversal}. For completeness we include its easy proof.
 \begin{lemma}
 \label{lem:Hruni}
 Let $r>s \ge 2$, $\Delta \ge 1$ be integers and $F$ be an $s$-uniform hypergraph on at most $r$ vertices.
 Further let $\cF$ be a family of $r$-uniform hypergraphs and $\cG$ a family of $s$-uniform hypergraphs hitting $\cF$ on $F$.
 If $G'$ is a $\cG$-universal $s$-uniform hypergraph, then $\cH_{(F,r)}(G')$ is $\cF$-universal.
 \end{lemma}
 \begin{proof}
  Let $H\in \cF$ be an $r$-uniform hypergraph together with the hypergraph $G \in \cG$ that hits $H$ on $F$. Since $G'$ is $\cG$-universal, there exists an embedding $\varphi\colon V(G)\to V(G')$ of $G$ into $G'$.

It is now easy to see that $\varphi$ is an embedding of $H$ into $\cH_{F,r}(G')$, and thus, $\cH_{F,r}(G')$ is $\cF$-universal. This can be seen as follows. For any edge $f\in E(H)$ there is a copy of $F$ in $G[f]$. Since $\varphi$ is an embedding of $G$ into $G'$, there is a copy of $F$ in $G'[\varphi(f)]$. By the definition of $\cH_{F,r}(G')$, $\varphi(f)$ is a hyperedge in $\cH_{F,r}(G')$. Thus, $\varphi$ is an embedding of $H$ into $\cH_{F,r}(G')$.
 \end{proof}

The lemma above suggests a way of obtaining $r$-universal hypergraphs out of hypergraphs of smaller uniformity. This will be exploited for particular choices of $F$ in the following sections.

\section{Proof of Theorem~\ref{theom:main}}
\label{sec:proof1}
In this section we provide proofs of the cases~\eqref{case:prime_case}--\eqref{case:odd_case} of Theorem~\ref{theom:main}.
\subsection{\texorpdfstring{Proofs of~\eqref{case:prime_case} and~\eqref{case:even_case}}{Proofs of (1) and (2)}}
\label{sec:primecase}
Let $r>r'\ge 2$ and $\Delta\ge 1$ be integers such that $r'\mid r$. We take $F$ to be the $r'$-uniform perfect matching on $r$ vertices (and thus with $r/r'$ edges). Let $H\in \FnD$. Since every vertex lies in at most $\Delta$ edges there is a graph $H^\prime \in \cF^{(r^\prime)}(n,\Delta)$ hitting $H$ on $F$. Such an $H'$ can be obtained from $H$ by replacing every edge $f$ of $H$ with an arbitrary perfect $r'$-uniform matching  on $f$. Therefore, $\cF^{(r^\prime)}(n,\Delta)$ hits $\FnD$ on $F$.

Now if $G'$ is an $\cF^{(r^\prime)} (n,\Delta)$-universal hypergraph then, by Lemma~\ref{lem:Hruni}, $\cH_{(F,r)}(G')$ is $\FnD$-universal. Moreover, since any collection of $r/r^\prime$ independent edges from $G'$ forms an $r$-edge in $\cH_{(F,r)}(G')$, we have $e(\cH_{(F,r)}(G')) \le e(G)^{r/r^\prime}$.

If there exists an $\mathcal{F}^{(r')}(n,\Delta)$-universal hypergraph with $O(n)$ vertices and $O(n^{r'-r'/\Delta})$ edges, then we immediately obtain an $\FnD$-universal hypergraph on $O(n)$ vertices and with 
\begin{align*}
O\left( (n^{r^\prime-r^\prime/\Delta})^{r/r^\prime} \right) = O(n^{r-r/\Delta})
\end{align*}
edges. 

By Theorem~\ref{theom:AC} there exist optimal explicitly constructible $\cF^{(2)}(n,\Delta)$-universal graphs on $O(n)$ vertices with $O(n^{2-2/\Delta})$ edges. This yields for even $r$ an explicitly constructible optimal $\FnD$-universal hypergraph with $O(n^{r-r/\Delta})$ edges. A similar argument applies also for the case of explicitly constructible $\cF^{(2)}(n,\Delta)$-universal graphs  on $n$ vertices with  $O(n^{2-2/\Delta} \ln^{4/\Delta} n)$  edges, giving $\FnD$-universal hypergraphs  on $n$ vertices with $O (n^{r-r/\Delta} \ln^{2r/\Delta}(n) )$ edges.

We remark, that obtaining $\cF^{(r')}(n,\Delta)$-universal hypergraphs on $O(n)$ vertices with $O(n^{r'-r'/\Delta})$ edges for $r'$ being prime would provide then the conjectured optimal upper bound $O(n^{r-r/\Delta})$ for all $r$ and $\Delta$.
%
%
%

\subsection{\texorpdfstring{Proof of~\eqref{case:odd_case}: finding low degree hitting graphs}{Proof of (3): finding low degree hitting graphs}}
\label{sec:lowdegree}
In the case when $r$ is odd, our hitting $r'$-uniform hypergraphs will be simply graphs, i.e.\ $r^\prime=2$. Moreover, the graph $F$ can no longer be perfect matching, and thus we take $F$ as the disjoint union of a matching on $r-3$ vertices and a path $P_3$ of length $2$, i.e.\ a path with $2$ edges. We remark, that the cases when $F=K_2$ (a single edge) and $F=K_r$  were considered in~\cite{PP_SparseUniversal}. We use the following lemma which asserts that one can find a family of graphs with not too large maximum degree which hits $\FnD$ on $F$.
\begin{lemma}
\label{lem:lowdegree}
Let $r \ge 3$ be odd  and $\Delta \ge 1$ be integers. Let $F$ be the disjoint union of a matching on $r-3$ vertices and a path $P_3$.  
Then $\cF^{(2)}(n,\lceil (r+1)\Delta/r \rceil)$ hits $\FnD$ on $F$.
\end{lemma}

\begin{proof}
Let $H \in \FnD$.  
 One defines an auxiliary bipartite incidence graph $B$ as follows. 
 The first class $V_1$ consists of  $\lceil \Delta / r \rceil$ copies of $V(H)$ and the second class  $V_2$ is equal to $E(H)$, while an edge of $B$  corresponds to a pair $(v,f)$, where  $v$ is some copy of a vertex from $V(H)$ and $f \in E(H)$ such that  $v \in f$.
The vertices in $V_1$ have degree at most  $\Delta$ and every hyperedge is connected to all $\lceil \Delta/r \rceil$ copies of its $r$ vertices, i.e.\ the vertices from $V_2$ have degree $r \lceil \Delta/r \rceil\ge \Delta$. By Hall's condition,  there is then a matching $M$ covering $V_2$ and thus of size $e(H)$.

We build the hitting graph $H^\prime$ on the vertex set $V(H)$ by replacing edges $f\in E(H)$ through copies of $F$ as follows. 
 For every edge $f$ in $E(H)$ we use the edge $(v,f)$ of the matching $M$ and place a copy of $F$ on $f$ such that the vertex $v$ is  the degree $2$ vertex of the path $P_3$ from $F$ while the other vertices are placed on $f\setminus\{v\}$ arbitrary. Since there are $\lceil \Delta/r \rceil$ copies of every vertex $v$ and every vertex $v$ lies in at most $\Delta$ edges of $H$, we see that each `placed' copy of $F$ that contains $v$ contributes  $1$ or $2$ to $\deg_{H'}(v)$. Thus,  the contribution from the copies of $F$ to the degree of $v$ in $H'$ is at most $\Delta + \left\lceil \frac{\Delta}{r} \right\rceil$ and therefore $\Delta(H')\le \lceil (r+1)\Delta/r \rceil$. This implies  $H^\prime \in \cF^{(2)}(n,\lceil (r+1)\Delta/r \rceil)$.
\end{proof}

For any $\cF^{(2)}(n,\lceil (r+1)\Delta/r \rceil)$-universal graph $G$ we get with Lemma~\ref{lem:Hruni} an $\FnD$-universal hypergraph $H$ with at most $2E(G)^{(r-1)/2} \Delta(E(G))$ many edges on the same number of vertices. The maximum degree of universal graphs $G$ in the constructions of Alon and Capalbo from Theorem~\ref{theom:AC} is  $O(|E(G)|/|V(G)|)$, and thus we obtain the case~\eqref{case:odd_case} of Theorem~\ref{theom:main} with $\cF^{(2)}(n,\lceil (r+1)\Delta/r \rceil)$-universal graph $G$ on $O(n)$ vertices with $O(n^{2-2/\lceil (r+1)\Delta/r \rceil})$ edges 
since
\begin{align*}
O\left((n^{2-2/\lceil (r+1)\Delta/r \rceil})^{(r-1)/2} \cdot n^{1-2/\lceil (r+1)\Delta/r \rceil}\right) = 
O\left(n^{r-(r+1)/\lceil (r+1)\Delta/r \rceil}\right). 
\end{align*}
A similar calculation yields $\FnD$-universal hypergraphs on $n$ vertices with 
\[O (n^{r-(r+1)/ \lceil (r+1)\Delta/r \rceil} \ln^{2(r+1)/\lceil (r+1)\Delta/r \rceil}n )\] edges, which we obtain from $\cF^{(2)}(n,\lceil (r+1)\Delta/r \rceil)$-universal graphs $G$ on $n$ vertices with $O(n^{2-2/\lceil (r+1)\Delta/r \rceil}\log^{4/\lceil (r+1)\Delta/r \rceil} n)$ edges.

\begin{remark}
In contrary to the $F$ chosen as a matching plus $P_3$ we could work with any forest $F$. 
For example, if $F$ is the path $P_r$ on $r$ vertices one can show that $\cF^{(2)}(n,\lceil 2 (r-1) \Delta/r\rceil)$ hits $\FnD$ on $F$. This leads to an $\FnD$-universal hypergraph on $O(n)$ vertices with  $O(n^{r-2(r-1)/\lceil 2 (r-1) \Delta/r\rceil})$ edges. It depends on the values of $r$ and $\Delta$, which bound is better, but one does not get anything significantly  better than $O\left(n^{r-(r+1)/\lceil (r+1)\Delta/r \rceil}\right)$ edges and therefore  we do not further pursue this here. 
\end{remark}

\subsection{Reducing the number of vertices}

Note that it is possible to reduce the number of vertices from $O(n)$ to $(1 + \eps) n$ in Theorems \ref{theom:AC}, \ref{theom:main} and \ref{theom:sec}, for any fixed $\eps>0$, by using a \emph{concentrator} as was done in \cite{ACKRRS01}.
 Consider the $\FnD$-universal hypergraph $H$ on $O(n)$ vertices and with $m$ edges. 
 A \emph{concentrator} is a bipartite graph $C$ on the vertex sets $V(H)$ and $Q$, where $|Q|=(1+\eps)n$ such that for every $S \subseteq V(H)$ with $|S| \le n$ we have $|N(S)| \ge |S|$ and every vertex from $V(H)$ has $O_\eps(1)$ neighbours in $C$.  We define a new hypergraph $H^\prime$ on $Q$ by taking all sets $f' \in \binom{Q}{r}$ as edges for which there exists a perfect matching in $C$ from an edge $f \in E(H)$ to $f'$. Since every vertex from $V(H)$ has $O_\eps(1)$ degree in $C$, the hypergraph $H'$ has $O_\eps(m)$ edges. It is also not difficult to see that $H'$ is $\FnD$-universal. Indeed, let $F\in\FnD$ and let $\varphi\colon V(F)\to V(H)$ be its embedding into $H$. By the property of the concentrator $C$,  there is a matching of $\varphi(V(F))$ in $C$ which we can describe by an injection $\psi\colon \varphi(V(F))\to V(H')$. But now, by construction of $H'$, $\psi\circ\varphi$ is an embedding of $F$ into $H'$.
 
\section{Proof of Theorem \ref{theom:sec}}
\label{sec:proof2}

At this point in all cases where $r$ is not even and $r$ does not divide $\Delta$ we do not have constructions of $\FnD$-universal hypergraphs that match the lower bound $\Omega(n^{r-r/\Delta})$ on the number of edges. In this section we will deal with the 'smallest' open case $\Delta=2$ by constructing optimal $\cF^{(r)}(n,2)$-universal hypergraphs on $O(n)$ vertices with $O(n^{r/2})$ edges. So, for example, if $r=3$ then Theorem~\ref{theom:main}, case~\eqref{case:odd_case} yields $\cF^{(3)}(n,2)$-universal hypergraphs on $O(n)$ vertices with $O(n^{3-4/\lceil8/3\rceil})=O(n^{5/3})$ edges, while the lower bound is  $\Omega (n^{3/2})$. 

We will first deal with the case $r=3$ and $\Delta=2$ and we reduce the case of general $r$ and $\Delta=2$ to this one.
Let us say a few words how an improvement from $O(n^{5/3})$ to $O(n^{3/2})$ can be accomplished.  We will use the concept of a graph $G$ that hits some hypergraph $H$ on $P_3$ (the path on $3$ vertices). If we would follow the arguments in the previous section, then we see that taking a hypergraph $H\in \cF^{({3})}(n,2)$ and replacing every hyperedge by $P_3$ we can obtain a hitting graph $F$ of maximum degree $3$ and of average degree $8/3$. Thus, if we would like to use Theorem~\ref{theom:AC} we need to consider $\cF^{(2)}(n,3)$-universal graphs, which results in the loss of some $n^{1/6}$-factor in the edge density. Instead, we will seek to decompose the hitting graph $F$ into appropriate subgraphs $F_1$, $F_2$, $F_3$ and $F_4$ such that every edge of $F$ lies in \emph{exactly} three of the graphs $F_i$. A decomposition result of Alon and Capalbo from~\cite{AC07} will assist us in this. 
Finally, following closely the arguments again due to Alon and Capalbo but now from~\cite{AC08} will allow us to construct a universal graph $G$ on $O(n)$ vertices and with maximum degree $O(n^{1/4})$ for a carefully chosen family $\cF'$ of graphs allowing a decomposition as above, which hits $\cF^{({3})}(n,2)$ on $P_3$. 
Lemma~\ref{lem:Hruni} implies then that $\cH_{P_3,3}(G)$ is $\cF^{({3})}(n,2)$-universal and has $O(n^{3/2})$ edges.

\subsection{A graph decomposition result}
The following notation is from~\cite{AC07}. 
Let $F$ be a graph and $S\subseteq V(F)$ be a subset of its vertices. A graph $F'$ which is obtained from $F$ by adding additionally $|S|$ new vertices to $F$ and placing an (arbitrary) matching between these new vertices and the vertices from $S$ is called an \emph{augmentation} of $F$. We call a graph \emph{thin} if every of its components is an augmentation of a path or a cycle, or if they contain at most two vertices of degree $3$. We also call any subgraph of a thin graph thin. 

The following decomposition theorem may be seen as a generalization of Petersen's theorem to graphs of odd degree. It was proved  in~\cite[Theorem~3.1]{AC07}.
\begin{theorem}
\label{lem:ACdecomposition}
Let  $\Delta$ be an integer and $F$ a graph with maximum degree $\Delta$. Then there are $\Delta$ spanning subgraphs $F_1,\dots,F_{\Delta}$ such that each $F_i$ is thin and every edge of $F$ appears in  precisely two graphs $F_i$.
\end{theorem}

Its proof is built on the Gallai-Edmonds decomposition theorem, and is implied by the following lemma.
\begin{lemma}[Lemma 3.3 from \cite{AC07}]
\label{lem:ACmatching}
Let $\Delta \ge 3$ be an odd integer and $H$ a $\Delta$-regular graph. Then $H$ contains a spanning subgraph in which every vertex has degree $2$ or $3$ and every connected component has at most $2$ vertices of degree $3$.
\end{lemma} 

We will use the two results above to prove the existence of a hitting graph $F$ with nice properties so that we can later take advantage of them when constructing a universal graph for the family of such `nice' hitting graphs.

\begin{lemma}
\label{lem:hitting}
Let $H\in \cF^{(3)}(n,2)$. Then there exists a graph $F$ that hits $H$ on $P_3$ with the following properties:
\begin{enumerate}
\item there are spanning subgraphs $F_1$, $F_2$, $F_3$ and $F_4$ of $F$ such that every $F_i$ is an augmentation of a thin graph, and\label{hit_three_one}
\item  every edge lies in \emph{exactly} three of the $F_i$.\label{hit_three_two}
\end{enumerate}
\end{lemma}
\begin{proof}
Let $H \in \cF^{(3)}(n,2)$. We assume first that $H$ is linear, i.e.\ edges are always intersecting in at most one vertex. 
Further we assume that $H$ is $2$-regular (otherwise we add `dummy' vertices and edges and obtain a $2$-regular hypergraph, and, once the desired graph $F$ is constructed, we delete these dummy vertices from $F$). 

The rough outline of the proof is to find a graph $F$ that hits $H$ on $P_3$ and such that $F$ contains a matching $M$ so that $F\setminus M$ is an augmentation of a thin graph and if we contract the matching edges from $M$ in $F$ we obtain a graph of maximum degree at most $3$. Decomposing such contracted graph via Theorem~\ref{lem:ACdecomposition} into thin graphs $F'_1$, $F'_2$ and $F'_3$ and then `recontracting' edges yields the desired family $F_1$,\ldots, $F_4$.


Let $H^*$ be the line graph of $H$, that is $V(H^*)=E(H)$ and $e\neq f\in E(H)$ form an edge $ef$  in $H^*$ if $e\cap f\neq\emptyset$. Thus, $H^*$ is a $3$-regular graph on $2n/3$ vertices.  Lemma \ref{lem:ACmatching} asserts then the existence of  a matching $M^*$ in $H^*$ such that in $H^* \setminus M^*$ every component has at most $2$ vertices of degree $3$ and all other vertices have degree $2$. Such a decomposition implies thus that every component of $H^* \setminus M^*$ is either a cycle, or has exactly two vertices, say $a$ and $b$, of degree $3$, so that either there are $3$ internally vertex-disjoint paths between $a$ and $b$ or there is one path between $a$ and $b$ and, additionally, $a$ and $b$ lie on vertex-disjoint cycles (which also do not contain inner vertices from the path between $a$ and $b$). 

From the  matching $M^*$ we define 
a subset $D:=\left\{v\colon e\cap f=\{v\} \text{ where }  ef\in E(M^*)\right\}$. Since $\Delta(H)\le 2$ and $M^*$ is a matching in the line graph of $H$ it follows that no two vertices from $D$ lie in an edge from  $H$.

We denote by $H_D$ the hypergraph which we obtain from $H$ if we delete from the edges of $H$ the vertices in $D$ but we keep the edges, obtaining thus a hypergraph on the vertex set  $V(H)\setminus D$, whose edges have cardinality $2$ or $3$. Thus, if $ef$ is an edge in $H^*$ and $e\cap f=\{v\}$ then the deletion of $v$ from $e$ and $f$ implies that the edges $e\setminus\{v\}$ and $f\setminus\{v\}$ are no longer adjacent in the line graph $(H_D)^*$, which corresponds to the deletion of the edge $ef$ in $H^*$. 
This implies that every component of $H^* \setminus M^*$ corresponds to a component of $H_D$, and therefore in every component of $H_D$ there are at most two edges of cardinality $3$ and all other edges have cardinality exactly $2$. Again, the structure of every component of $H_D$ is thus either a (graph) cycle, or there are exactly  two edges, say $g$ and $h$, of cardinality $3$, where $g\cap h=\emptyset$ and there are three  vertex-disjoint (graph) paths that connect the vertices from $g\cup h$ or $g\cap h\neq\emptyset$ and there are two vertex-disjoint (graph) paths that connect the vertices from $g\Delta h$.

Finally we come to the definition of the hitting graph $F$. For every component $C$ of $H_D$, let $D_C$ be the vertices that have been deleted from the hyperedges in $H$ that lie now in $H_D$. Thus, there is a (natural) map $\psi_C$ between the edges from $C$ of cardinality $2$ and $D_C$: $\psi_C(f)=v$  if $\{v\}\cup f\in E(H)$. Since every vertex from $D$ lies in exactly two edges of $H$, it will suffice to explain how we replace the $3$-uniform edges of $H_D$ and the edges of $H$  incident with $D$ by paths $P_3$.  
If $C$ is the (graph) cycle, then we replace every edge of the form $\{v\}\cup f$, where $\psi_C(f)=v$, by $P_3$ so that the graph $F_C$ obtained contains all the edges from $E({C})$ and  is such that $\Delta(F_C)\le 3$ and the vertices from $D_C$ have degree at most $2$ in $F_C$. 
If $C$ contains exactly two $3$-uniform edges (say  $g$ and $h$), then it is possible to replace the edges $g$, $h$ and every edge of the form $\{v\}\cup f$, where $\psi_C(f)=v$, by $P_3$ such that the graph $F_C$ satisfys the following: 
It contains all $2$-uniform edges of $C$, is such that $\Delta(F_C)\le 3$, the vertices from $D_C$ have degree at most $2$ in $F_C$ and $F_C\setminus D_C$ is connected and has exactly two vertices of degree $3$ (this is easily done by considering the structure of the components $C$ from $H_D$ described in the previous paragraph). 
The graph $F$ is then the union of all $F_C$ and observe that $F_C$ and $F_{C'}$ intersect in $D_C\cap D_{C'}$ for $C\neq C'$ and in particular have no common edges. 
Furthermore, every vertex from $D$ has degree $2$ in $F$.


 Let $M$ be a matching in $F$ that saturates $D$. Such a matching exists since $D$ is independent in $F$ (no two vertices from $D$ lie in an edge from $H$), every vertex of $D$ is connected to a vertex of degree $2$ in $F\setminus D$ and $\deg(F)\le 3$. By the definition of $F$  above, every component in $F \setminus M$ is an augmentation of a graph with at most two vertices of degree $3$, and thus an augmentation of a thin graph. We set $F_4:=F\setminus M$.  Next  we contract the edges of $M$ in $F$ obtaining the graph $F/M$. Since $M$ saturates $D$, which are vertices of degree $2$ in $F$, it follows that $F/M$ has maximum degree at most $3$. Theorem~\ref{lem:ACdecomposition} yields a decomposition of $F/M$ into thin graphs $F_1^\prime$, $F_2^\prime$, $F_3^\prime$ such that every edge of $F/M$ appears in precisely two of the graphs. Now we reverse the recontraction procedure. This leads to three graphs $F_1$, $F_2$ and $F_3$ where every edge of $E(F)\setminus M$ appears in exactly two of the graphs, every edge from $M$ appears in all three of them, and each of the $F_1$, $F_2$ and $F_3$ is an augmentation of a thin graph. Together with the graph $F_4=F\setminus M$ we thus constructed the desired decomposition of a hitting graph $F$. 

If $H$ is not linear, then things get in some sense even easier, so we shall be brief. We proceed essentially in the same way. That is, we define the line graph $H^*$ of $H$, which is now not necessarily $3$-regular, but whose maximum degree is at most $3$. Again, Lemma \ref{lem:ACmatching} asserts then the existence of  a matching $M^*$ in $H^*$ such that in $H^* \setminus M^*$ every component has at most $2$ vertices of degree $3$ and all other vertices have degree at most $2$. We then define the set $D$ as before but in the case that the edge $ef\in M^*$ with, say, $e=\{a,b,c\}$ and $f=\{b,c,d\}$ we simply replace the edge $e$ by $\{a,b\}$ and $f$ by $\{c,d\}$ without putting anything into $D$. Once the components of $H_D$ are identified and the graphs $F_C$ are defined we add the edge $bc$ (which we call nonlinear) to those graphs $F_C$, which contain either $b$ or $c$ (or both). Then we choose edges into the matching $M$ as before and add all nonlinear edges such as $bc$ to $M$. The rest of the argument remains the same.

\end{proof}

An $\ell$-th power of a graph $F$, denoted by $F^\ell$, is the graph on $V(F)$, whose vertices at distance at most $\ell$ in $F$ are connected. It is not difficult to see that a thin graph on $n$ vertices can be embedded into $P_n^4$, and thus, an augmentation of a thin graph into $P^8_n$. This motivates the following general definition. 

\begin{definition}[$(k,r,\ell)$-decomposable graphs]
Let $k$, $r$ and $\ell$ be integers.
A graph $F$ on $n$ vertices is called \emph{$(k,r,\ell)$-decomposable} if there exist $k$ graphs $F_i$ with the following properties. Every edge of $F$ appears in exactly $r$ of the $F_i$ 
and there are maps $g_i\colon F_i \to [n]$, which are injective homomoprhisms from $F_i$ into $P_{n}^\ell$. 
Then we denote by $\cF_{k,r,\ell}(n)$ the family of $(k,r,\ell)$-decomposable graphs on $n$ vertices. 
\end{definition}

We can restate our Lemma~\ref{lem:hitting} in the following slightly weaker form. 
\begin{lemma}\label{lem:decomposable}
The family $\cF_{4,3,8}(n)$ hits $\cF^{(3)}(n,2)$ on a path $P_3$.
\end{lemma}
This lemma implies that it is the family $\cF_{4,3,8}(n)$ for which a universal graph is needed. 
This graph will be constructed in the subsection below and briefly explained why a desired embedding works, which will follow from the results of  Alon and Capalbo from~\cite{AC08}.

\subsection{Constructions of universal graphs}
First we briefly describe the construction from~\cite{AC08} of $\cF^{(2)}(n,k)$-universal graphs on $O(n)$ vertices with $O(n^{2-2/k})$ edges. One chooses  $m=20n^{1/k}$, a fixed $d>720$ and a graph $R$ to be a $d$-regular graph on $m$ vertices with the absolute value of all but the largest eigenvalues at most $\lambda$ (such graphs are called $(n,d,\lambda)$-graphs). One can assume that $\lambda\le 2\sqrt{d-1}$ (then $R$ is called Ramanujan) and $\mathrm{girth}({R})\ge \tfrac{2}{3}\log m/\log (d-1)$. Explicit constructions of such Ramanujan graphs have been found first for $d-1$ being a prime congruent to $1$ mod $4$ in~\cite{lubotzky1988ramanujan,Ma88expander}.  Finally, the graph $G_{k,n}$ is defined on the vertex set $V({R})^k$ where two vertices $(x_1,\dots,x_k)$ and $(y_1,\dots,y_k)$ are adjacent if and only if there are at least two indices $i$  such that $x_i$ and $y_i$ are within distance $4$ in $R$. It is easily seen that such a graph $G_{k,n}$ has $O(n)$ vertices, $O(n^{2-2/k})$ edges and maximum degree $O(n^{1-2/k})$. 

The first step in the proof of $\cF^{(2)}(n,k)$-universality of $G_{k,n}$ is Theorem~\ref{lem:ACdecomposition} implying that any graph $F$ with $\Delta(F)\le k$ is $(k,2,4)$-decomposable. In what follows we summarize a straightforward generalization of the central claim from~\cite{AC08} (which is inequality~(3.1) there), from which an  existence of embedding of any graph $F\in \cF^{(2)}(n,k)$ into $G_{k,n}$ follows. Its proof can be taken almost verbatim from~\cite{AC08}.
\begin{lemma}
\label{lem:AChom}
Let $k\ge 3$, $r$ and $\ell$ be natural numbers. 
For any choice of $k$ permutations $g_i \colon [n] \to [n]$ there are $k$ homomorphisms $f_i \colon [n] \to V(F)$ from the path $P_n$ to the above Ramanujan graph $R$ such that the map $f \colon [n] \to V(G_{k,r,\ell}(n))$ defined by $f(v) =(f_1(g_1(v)),\dots,f_k(g_k(v)))$  is injective.
\end{lemma}
More precisely, the $f_i$'s are inductively constructed as non-returning walks preserving the property that for any $i$ vertices $v_1,\dots,v_i \in V(F)$ one has
\begin{align*}
|\{ v \in [n] : f_1(g_1(v))=v_1, \dots, f_i(g_i(v))=v_i \}| \le n^{(k-i)/k}.
\end{align*}
For the last step $i=k$ this is equivalent to injectivity. 

Finally, we explain, how we obtain $\cF_{k,r,\ell}(n)$-universal graphs. 
The choice of the Ramanujan graph $R$ along with the parameters $m$ and $d$ remains the same. 
The graph $G_{k,r,\ell}(n)$ is defined on the vertex set $V({R})^k$ and  two vertices $(x_1,\dots,x_k)$ and $(y_1,\dots,y_k)$ are adjacent if and only if there are at least $r$ indices $i$ such that $x_i$ and $y_i$ are within distance $\ell$ in $R$. 
It is then an easy calculation to show  that $G_{k,r,\ell}(n)$ has $O(n)$ vertices, at most $n\binom{k}{r}d^{r\ell}m^{k-r}=O(n^{2-r/k})$ edges and maximum degree $O(n^{1-r/k})$, where the constants in $O$-notation depend on $k$, $r$, $\ell$ and $d$. 
Lemma~\ref{lem:AChom} implies then the following.

\begin{theorem}
\label{theom:universal}
Let $k\ge 3$, $r$ and $\ell$ be natural numbers. The graph $G_{k,r,\ell}(n)$ is $\cF_{k,r,\ell}(n)$-universal.
\end{theorem}
\begin{proof}
Let $F$ be a $(k,r,\ell)$-decomposable graph on $n$ vertices together with the decomposition $F_1,\dots,F_k$ and injective homomorphisms $g_i \colon V(F_i) \to [n]$ from $F_i$ into $P_n^\ell$. Lemma~\ref{lem:AChom} asserts the existence of the homomorphisms 
 $f_i\colon [n] \to V({R})$  from  $P_n$ to $R$ for every $i\in[k]$, so that  the map $f \colon V(F) \to V(G_{k,r,\ell}(n))$ given by $f(v) = (f_1(g_1(v)), \dots, f_k(g_k(v)))$ is injective.

It is clear that  the composition of $f_i$ with $g_i$ is a homomorphism from $F_i$ to $R^\ell$.
Furthermore, every edge $\{ u,v \}$ from $F$ lies in $r$ graphs $F_i$. Thus, there are $r$ indices $i$ such that $g_i(u)$ and $g_i(v)$ are distinct and within distance $\ell$ in $P_n$.
This implies that $f_i(g_i(u))$ and $f_i(g_i(v))$ are also distinct and within distance $\ell$ in $F$.
By the definition of  $G_{k,r,\ell}(n)$ this implies that $f(u)$ and $f(v)$ are adjacent in  $G_{k,r,\ell}(n)$  and $f$ is the desired embedding of $F$ into  $G_{k,r,\ell}(n)$.
\end{proof}

From this, Theorem~\ref{theom:sec} follows immediately for $r=3$.
\begin{proof}[Proof of Theorem \ref{theom:sec}, case $r=3$]
Note, that the graph $G_{4,3,8}(n)$ has $m^4=O(n)$ vertices and $O(n m) = O(n^{5/4})$ edges. 
By Theorem~\ref{theom:universal} $G_{4,3,8}(n)$ is $\cF_{4,3,8}(n)$-universal,  and since $\cF_{4,3,8}(n)$ hits $\cF^{(3)}(n,2)$ on $P_3$, Lemma~\ref{lem:Hruni} implies that $\cH_{P_3,3}(G_{4,3,8}(n))$ is $\cF^{(3)}(n,2)$-universal, has $O(n)$ vertices and $O(n^{3/2})$ edges. 
This proves the case $r=3$.
\end{proof}

\begin{remark}
We believe that the constructions from \cite{AC07} can also be adapted to work with $(k,r,\ell)$-decomposable graphs. For the cases discussed here this would lead to universal graphs on $n$ vertices, where the number of edges is some polylog factor larger.
\end{remark}

\subsection{\texorpdfstring{Proof of Theorem \ref{theom:sec} for odd $r\ge 5$}{Proof of Theorem~1.3 for odd r ≥ 5}}
\subsubsection{A universal hypergraph}
First we define the hypergraph $\cH$ which will turn out to be  $\cF^{({r})}(n,2)$-universal. Let $t=(r-3)/2$. 
Let $G_1$,\ldots, $G_{t+1}$ be vertex-disjoint graphs,  where $G_1$, \ldots, $G_t$ are copies of  $C^4_n$ (the fourth power of the cycle $C_n$) and $G_{t+1}$ is a copy of the graph $G_{4,3,8}(n)$, introduced in the previous section. We define $\cH$ to be the hypergraph on the vertex set $\dcup_{i=1}^{t+1} V(G_i)$, and the edges are $r$-element subsets $f$ such that, with $f_i:=f\cap V(G_i)$, each $G_{i}[f_i]$ contains a copy of $P_{|f_i|}$, a path on $|f_i|$ vertices (thus, $P_0$ is the empty graph, $P_1=K_1$ and $P_2=K_2$). Certainly, $\cH$ has $O(n)$ vertices. How many edges does the hypergraph $\cH$ contain? For this we need to choose paths $P_{\ell_i}$ from every $G_i$ such that $\ell_i\in\{0,1,2,3\}$ and $\sum_{i=1}^{t+1}\ell_i=r$. Because $G_1$, \ldots, $G_t$ have maximum degree $8$, and $G_{t+1}$ has maximum degree $O(n^{1/4})$, we compute the number of edges of $\cH$ to be $O(n^{t+1}n^{2/4})=O(n^{r/2})$, as desired.

Given any hypergraph $H\in \cF^{({r})}(n,2)$, we show that one can partition its vertex set into disjoint subsets $X_1$, \ldots, $X_{t+1}$ and, for each $i$, define the hypergraph $H(X_i)$ -- the (not necessarily uniform) hypergraph on the vertex set $X_i$ -- whose edges are restrictions to $X_i$, i.e.\ $E(H(X_i))=\{f\cap X_i\colon f\in E(H)\}$. The hypergraphs $H(X_1)$, \ldots, $H(X_t)$ have simple structure --- each component contains at most two edges of cardinality $3$. The hypergraph $H(X_{t+1})$ contains hyperedges of cardinality at most $3$. Moreover, these hypergraphs have maximum vertex degree at most $2$. We defer this structural decomposition to the subsection below. 

Let us see, how then  $H$ can be embedded into the hypergraph $\cH$. Owing to the structure of $H(X_1)$, 
\ldots, $H(X_t)$, one can easily find injective maps $g_i\colon X_i\to V(G_i)$, such that every hyperedge $f\in E(H(X_i))$ is such that $G_i\left[g_i(f)\right]$ contains a path $P_{|f|}$ -- this can be seen by replacing $f$  in $H(X_i)$ through an arbitrary path $P_{|f|}$ and finding, due to the component structure of $H(X_i)$, an injective graph homomorphism from the appropriately defined graph into $G_i$. 
As for $H(X_{t+1})$ we can assume first that it is $3$-uniform and in $\cF^{({3})}(n,2)$ by adding some `dummy' vertices appropriately (but using the same notation for $H(X_i)$). The $\cF_{4,3,8}(n)$-universality of $G_{t+1}=G_{4,3,8}(n)$ and the fact that $\cF_{4,3,8}(n)$ hits $\{H(X_{t+1})\}$ on $P_3$ yields an injective map $g_{t+1}\colon X_{t+1}\to V(G_{t+1})$ such that $G_{t+1}\left[g_{t+1}(f)\right]$ contains $P_3$ for every $f\in E(H(X_{t+1}))$. 
Deleting the dummy vertices from the edges $f\in E(H(X_{t+1}))$ (but keeping the edges) we see that $g_{t+1}$ remains injective and  $G_{t+1}\left[g_{t+1}(f)\right]$ contains $P_{|f|}$ for every $f\in E(H(X_{t+1}))$. 
It should be clear that $g\colon V(H)\to V(\cH)$ with $g|_{X_i}=g_i$, for all $i\in[t+1]$, is injective. It remains to show that it is a homomorphism into $\cH$. Given an edge $e$ of $H$, by the definition of $H(X_i)$ and the choices of $g_i$'s, we see that $e\cap X_i\in E(H(X_i))$ and $G_i\left[g_i(e\cap X_i)\right]$ contains a path $P_{|e\cap X_i|}$ for all $i$. But this is exactly the requirement for $g(e)$ to be the edge in $\cH$. Thus, $g$ embeds $H$ into $\cH$. 

\subsubsection{Yet another decomposition}
 Let $H\in \cF^{({r})}(n,2)$. Again we assume first that $H$ is linear and $2$-regular. We consider, as in the case $r=3$, the line graph $H^*$, which is $r$-regular now. Hence Lemma~\ref{lem:ACmatching} yields a spanning subgraph $H^*_1$, in which every vertex has degree $2$ or $3$ and every component has at most $2$ vertices of degree $3$. 

If $C$ is a component of $H^*_1$, then we define $V_C$ as all vertices $v$ such that $\{v\}=e\cap f$ for some $ef\in E({C})$ (recall that $H$ is assumed to be a linear hypergraph). We write $H_1$ to denote the hypergraph on the vertex set $\cup V_C$ where the union is over all components $C$ of $H^*_1$, and for every edge $f\in E(H)$ let the set $\left\{v\colon \{v\}=e\cap f \text{ for some }ef\in E({C})\right\}$ be an edge of $H_1$. Observe, that these edges have cardinality  either $2$ or $3$. Indeed,  a vertex of degree $j$ in some component $C$ is the edge of $H$ that intersects $j$ other edges of $H$ in different vertices, which give rise to a $j$-uniform edge in $H_1$. By construction, $H_1$ is linear and $2$-regular. Crucially, the components of $H_1$ have simple structure, since these are `inherited' from the components  $C$. More precisely, each component of $H_1$ has at most two $3$-uniform edges and all other edges have cardinality $2$.

We denote $\tH_{1}$ as the hypergraph obtained by deleting from its edges all vertices from $V_C$ (we call this procedure as `reducing uniformity'). It should be clear that, in this way every edge of $H$ can be written uniquely as the union of one edge of $H_1$ and the other from $\tH_1$. Since $H_1$ is not necessarily uniform, the hypergraph $\tH_{1}$ is now a not necessarily uniform hypergraph as well, but its edges have cardinalities  either $r-3$ or $r-2$. 

The next step calls for an inductive procedure with a blemish, that $\tH_{1}$ is not necessarily uniform. But this can be remedied by adding `dummy' vertices and edges to $\tH_{1}$ and obtaining an $(r-2)$-uniform linear hypergraph still denoted by $\tH_{1}$ which is $2$-regular (once we are finished with decomposition, we will reduce the uniformity by deleting these dummy vertices from edges, but keeping the altered edges). We keep doing this reduction until we arrive at the hypergraph $\tH_{t}$ where $t=(r-3)/2$, and thus $H_{t+1}:=\tH_t$ is a $3$-uniform linear hypergraph, which is $2$-regular. 

Before we proceed, let us summarize what we achieved so far. We have found hypergraphs $H_1$,\ldots, $H_t$ and the 
hypergraph $H_{t+1}$ with pairwise disjoint vertex sets, so that each $H_i$ is linear, $2$-regular and its edge uniformities are either $2$ or $3$ and each component of $H_i$ has simple structure (recall:  each component has at most two $3$-uniform edges and all other edges have cardinality $2$), while $H_{t+1}$ is a $3$-uniform linear hypergraph, which is $2$-regular. 

 Next we reduce uniformities of the $H_i$ and and of $H_{t+1}$ by deleting dummy vertices from the edges. 
 In this way it may happen, that the uniformity of some edges of the hypergraph family will be reduced to $0$ (in which case they disappear from that particular hypergraph), while some others will be reduced to $1$, in which case we get edges of the type $\{v\}$, which we will use. We finally obtain the promised family $H_1$,\ldots, $H_{t+1}$. 

The case when $H$ is not a linear hypergraph can be treated similarly and we omit the details.
\subsection{A general problem}

To prove the embedding for other parameters of $r$ and $\Delta$ we would need the analogue of Lemma~\ref{lem:decomposable}, that  is, a solution to the following problem.

\begin{problem}\label{problem}
Let  $r \ge 3$ and $\Delta \ge 3$ be integers.  Find $\ell$ such that $\cF_{(r-1)\Delta,r,\ell}(n)$ hits $\FnD$ on $P_{r}$.
\end{problem}

It is immediate that Theorem~\ref{theom:universal}   yields $\cF_{(r-1)\Delta,r,\ell}(n)$-universal graphs $G$ on $O(n)$ vertices with $O(nm^{(r-1)\Delta-r})=O(n^{2-r/((r-1)\Delta)})$ edges and maximum degree $O(n^{1-r/((r-1)\Delta)})$. From this the solution to  Problem~\ref{problem} would yield optimal universal hypergraphs on $O(n)$ vertices with 
$|V(G)|(|E(G)|/|V(G)|)^{r-1}=O(n^{r-r/\Delta})$ edges. Clearly, the interesting cases are $\Delta\ge 3$, $r\nmid \Delta$ and $r$ odd.

\begin{remark}
An alternative to our approach is to extend the constructions for universal graphs from \cite{AC07,AC08,ACKRRS01} to hypergraphs.
To follow a similar embedding scheme one would ask for appropiate decomposition results for hypergraphs.
For example, for $H \in \cF^{(3)}(n,2)$ the task is to find subhypergraphs $H_1,\dots,H_4$ which are `thin' and such that every hyperedge appears in exactly three of them.
\end{remark}

 \section{Proof of Theorem~\ref{thm:Erm}}
\begin{proof}[Proof of Theorem~\ref{thm:Erm}]
 To prove the existence of optimal $\Erm$-universal hypergraphs we exploit the proof of Alon and Asodi~\cite{AA02}. 
Take any $H \in \Erm$ and replace all edges of $H$ by cliques of size $r$. 
This gives a graph with at most $\binom{r,2}m$ edges and thus there exists a graph $G$ with $O(m^2/\ln^2m)$ edges which is $\cE^{(2)}(\binom{r,2}m)$-universal. 
We define the $r$-uniform hypergraph $\cK_r(G)$ on the vertex set $V(G)$ with edges being the vertex sets of the copies of $K_r$ in $G$. It is straightforward to see that $\cK_r(G)$ is $\Erm$-universal and thus it remains to estimate the number of edges in $\cK_r(G)$.
 
 The $\cE^{(2)}(m)$-universal graph  $G$ of Alon and Asodi~\cite{AA02} is defined on the vertex set $V=V_0 \cup V_1 \cup \dots \cup V_k$ where $k= \lceil \log_2 \log_2 m \rceil$, $|V_0|=4m/\log_2^2m$ and $|V_i|=4m2^i/\log_2 m$ for $i \in [k]$.
 A vertex in $V_0$ is connected to any other vertex and the graph induced on $V_1$ is a clique.
For any $u \in V_i$, $i \ge 2$, and $v \in V_1 \cup V_2 \cup \dots \cup V_i$ with $u \not= v$ the edge $uv$ is present independently with probability $\min\left( 1 , 8^{3-i} \right)$. It is shown in~\cite{AA02} that with probability at least $1/4$ the graph $G$ has $O(m^2/\ln^2m)$ edges and is $\cE^{(2)}(m)$-universal. We count the expected number of copies of $K_r$ in $G$, i.e.\  $\EE (|E(\cK_r(G))| )$.
 
There are several possible types of cliques $K_r$ in $G$. Indeed, we need to choose $r$ vertices from $V_0$,\ldots,$V_k$, and a particular \emph{type} of a possible $r$-clique $K$ in $G$ is specified by $\alpha$, which is the number of its vertices in $V_0$ and by numbers $t_1\le \ldots\le t_\gamma$ (all from $[k]$), which specify to which sets $V_i$ the remaining $\gamma=r-\alpha$ vertices belong to.  There are at most $|V_0|^\alpha \prod_{j=1}^{\gamma} |V_{t_j}|$ cliques of a particular type, and each such clique occurs with probability  $\prod_{j=1}^{\gamma}\left[\min\left(1 , 8^{3-t_j}\right)\right]^{j-1}$. It is clear that there are at most $|V_0|^{r-1}|V(G)|\le \frac{(4m)^{r-1}\cdot (32 m)}{(\log_2 m)^{2(r-1)}}=o\left(\frac{m^r}{\log_2^{r} m}\right)$ cliques $K_r$ in $G$ that intersect $V_0$ in at least $r-1$ vertices. 
 Next we upper bound the expected number of edges in $\cK_r(G)$ as follows:
 {\allowdisplaybreaks
 \begin{multline}\label{eq:expectation_AA}
 \EE (|E(\cK_r(G))|) \le |V_0|^{r-1}|V(G)|+ \sum_{\substack{\alpha+\gamma = r\\ \gamma\ge 2}} \sum_{1 \le t_1 \le \dots \le t_\gamma \le k} |V_0|^\alpha \prod_{j=1}^\gamma |V_{t_j}| \cdot \prod_{j=1}^{\gamma}\left[\min\left(1 , 8^{3-t_j}\right)\right]^{j-1} \\
\le o\left(\frac{m^r}{\log_2^{r} m}\right)+ \sum_{\gamma\ge 2}^r  \left( \frac{4m}{\log_2m} \right)^r \frac{1}{\log_2^{r-\gamma} m} \sum_{1 \le t_1 \le \dots \le t_\gamma \le k} 2^{\sum_{j=1}^\gamma t_j }\cdot 2^{\sum_{j=1}^\gamma \min\left\{0,(9-3t_j)(j-1)\right\}}, 
 \end{multline}}
and in order to simplify it further we first estimate the inner sum of the second summand by splitting it according to $t_1$ as follows:{\allowdisplaybreaks
\begin{align*}
&\sum_{1\le t_1 \le \dots \le t_\gamma \le k} 2^{\sum_{j=1}^\gamma t_j }\cdot 2^{\sum_{j=1}^\gamma \min\left\{0,(9-3t_j)(j-1)\right\}}\\
&\le \sum_{ t_1\le 19}\sum_{ \substack{t_j\ge 1\\ j=2,\ldots,\gamma}} 2^{\sum_{j=1}^\gamma t_j +\sum_{j=1}^\gamma \min\left\{0,(9-3t_j)(j-1)\right\}}+
\sum_{ t_1\ge 20}\sum_{ \substack{t_j\ge t_1\\ j=2,\ldots,\gamma}} 2^{\sum_{j=1}^\gamma t_j +\sum_{j=1}^\gamma \min\left\{0,(9-3t_j)(j-1)\right\}}\\
&\le 2^{20}\sum_{ \substack{t_j\ge 1\\ j=2,\ldots,\gamma}} 2^{\sum_{j=2}^\gamma \left( t_j + \min\left\{0,(9-3t_j)(j-1)\right\}\right)}+
\sum_{t_1\ge 20}2^{t_1}\sum_{ \substack{t_j\ge t_1\\ j=2,\ldots,\gamma}} 2^{\sum_{j=2}^\gamma \left( t_j +(9-3t_j)(j-1)\right)}\\
&\le 2^{20}\left(\sum_{t\ge 1} 2^{ t + \min\left\{0,(9-3t)\right\}}\right)^{\gamma-1}+\sum_{ t_1\ge 20}2^{t_1}\left(\sum_{t\ge t_1} 2^{t+(9-3t)}\right)^{\gamma-1} \\
&\begin{aligned}
\le 2^{20}\left(6+\sum_{t\ge 3}2^{9-2t}\right)^{\gamma-1}+\sum_{ t_1\ge 20}2^{t_1}\left(\sum_{t\ge t_1} 2^{-3t/2}\right)^{\gamma-1}\le 2^{20+5\gamma}+\sum_{ t_1\ge 20}2^{t_1-\frac{3t_1(\gamma-1)}{2} +2(\gamma-1)}\\
\le 2^{20+5\gamma}+ 2^{2(\gamma-1)}\sum_{ t_1\ge 20}2^{-t_1/2}\le 2^{21+5\gamma}\le 2^{21+5r}.
\end{aligned}
\end{align*}}
 This allows us to further upper bound~\eqref{eq:expectation_AA} by
\[
 \EE (|E(\cK_r(G))|) \le r2^{21+5r} \left( \frac{4m}{\log_2m} \right)^r.
\]
By Markov's inequality, the probability that $|E(\cK_r(G))|$ is at least $5r2^{21+5r} \left( \frac{4m}{\log_2m} \right)^r$ is at most $1/5$. Thus, taking $\hat{m}=\binom{r}{2}m$, there exists an $\cE^{(2)}(\hat{m})$-universal graph with $O\left( \frac{\hat{m}^r}{\log^r_2 \hat{m}} \right)$ copies of $K_r$. This implies that there exists an $\Erm$-universal hypergraph $H$ with $O(m^r/\ln^rm)$ edges. 
\end{proof}

It is possible to prove that there exist such hypergraphs $H$ with  $rm$ vertices which is optimal. However,  no explicit construction is known. 

\subsection*{Acknowledgements}The authors would like to thank Rajko Nenadov for useful discussions, and the referees for helpful suggestions.
 

\providecommand{\bysame}{\leavevmode\hbox to3em{\hrulefill}\thinspace}
\providecommand{\MR}{\relax\ifhmode\unskip\space\fi MR }
\providecommand{\MRhref}[2]{%
  \href{http://www.ams.org/mathscinet-getitem?mr=#1}{#2}
}
\providecommand{\href}[2]{#2}
\def\MR#1{\relax}

\end{document}